\documentclass[12pt]{amsart}
\usepackage{amsmath,amsfonts,amsthm,amscd,amssymb,mathrsfs,amssymb, multirow}
\usepackage{tikz-cd}
\usepackage{stmaryrd}
\usepackage{graphicx}
\usepackage{verbatim}
\newtheorem{theorem}{Theorem}

\newtheorem{proposition}[theorem]{Proposition}
\newtheorem{lemma}[theorem]{Lemma}
\newtheorem{definition}[theorem]{Definition}

\numberwithin{figure}{section}
\newtheorem{remark}[theorem]{Remark}

\renewcommand{\H}{\mathcal{H}}

\newcommand{\CC}{\mathbb{C}}

\newcommand{\RR}{\mathbb{R}}
\newcommand{\ZZ}{\mathbb{Z}}

\newcommand{\WW}{\mathbb{W}}
\newcommand{\VV}{\mathbb{V}}
\newcommand{\U}{{\rm{U}}}

\newcommand{\oo}{\mathcal{O}}

\numberwithin{equation}{section}
\numberwithin{theorem}{section}
\numberwithin{table}{section}
\numberwithin{table}{section}
\begin{document}
\bibliographystyle{amsalpha} 
\title[Local moduli of scalar-flat K\"ahler ALE surfaces]{Local moduli of scalar-flat K\"ahler ALE surfaces}
\author{Jiyuan Han}
\address{Department of Mathematics, Purdue University, West Lafayette,  
IN, 47907}
\email{han556@purdue.edu}
\author{Jeff A. Viaclovsky}
\address{Department of Mathematics, University of California, Irvine, 
CA, 92697}
\email{jviaclov@uci.edu}
\thanks{The authors were partially supported by NSF Grant DMS-1811096.}
\date{September 14, 2018}
\begin{abstract} 
In this article, we give a survey of our construction of a local moduli space of scalar-flat K\"ahler ALE metrics in complex dimension $2$. We also prove an explicit formula for the dimension of this moduli space on a scalar-flat K\"ahler ALE surface which deforms to the minimal resolution of 
$\CC^2/\Gamma$, where $\Gamma$ is a finite subgroup of ${\rm{U}}(2)$ without complex reflections, in terms of the embedding dimension of the singularity. 
 \end{abstract}
\maketitle
\section{Introduction}
\label{intro}
In this article, the main objects of interest will be a certain class of complete non-compact K\"ahler metrics. In the following,  $\Gamma$ will always be a finite subgroup of ${\rm{U}}(2)$ containing no complex reflections.
\begin{definition}
Let  $(X, g,J)$ be a K\"ahler surface $(X, g,J)$ of complex dimension $2$, with metric $g$ and complex structure $J$.
We say that $(X,g,J)$ is asymptotically locally Euclidean (ALE) if there exists a compact subset $K\subset X$, a real number $\mu >0$, and a diffeomorphism 
$\psi: X \setminus K\rightarrow (\RR^4\setminus \overline{B})/\Gamma$, such that for each multi-index $\mathcal{I}$ of order~$|\mathcal{I}|$
\begin{align}
\partial^{\mathcal{I}}(\psi_*(g)-g_{Euc}) = O(r^{-\mu-|\mathcal{I}|}),
\end{align}
as $r \rightarrow \infty$. 
In the above, $B$ denotes a ball centered at the origin, and $g_{Euc}$ denotes the Euclidean metric. 
\end{definition}
The number $\mu$ is referred to as the order of $g$. It was shown in \cite{HL15} that for any ALE K\"ahler metric of order $\mu$, there exist ALE coordinates for which 
\begin{align}
\partial^{\mathcal{I}} (J - J_{Euc})= O(r^{-\mu-|\mathcal{I}|}),
\end{align}
for any multi-index $\mathcal{I}$ as $r \rightarrow \infty$, where $J_{Euc}$ is the standard complex structure on Euclidean space. This follows because the K\"ahler assumption implies that $J$ is parallel.

In this definition, we only assumed that the metric is K\"ahler. 
A natural condition is that the metric be in addition scalar-flat. 
Such metrics are then {\it{extremal}} in the sense of Calabi \cite{Calabi85}.
These spaces arise naturally as ``bubbles'' in orbifold compactness theorems for  sequences of extremal K\"ahler metrics 
\cite{Anderson, BKN, CW, CLW, Nakajimaconv, Tian, TV, TV2, TV3}. 
Furthermore, they arise in a number of natural  gluing constructions for extremal K\"ahler metrics \cite{ALM_kummer, ALM_resolution, Arezzo06, APSinger, BiquardRollin, GaborI, RollinSinger, RS09}. 

 We note that in the case of scalar-flat K\"ahler ALE metrics, it is known that there exists an ALE coordinate system for which the order of such a metric is at least $2$ \cite{LeBrunMaskit}. 

There are many known examples of scalar-flat K\"ahler ALE metrics:

\begin{itemize}

\item ${\rm{SU}}(2)$ case: when $\Gamma\subset {\rm{SU}}(2)$, Kronheimer has constructed families of hyperk\"ahler ALE metrics \cite{Kr89} on manifolds diffeomorphic to the 
minimal resolution of $\CC^2 / \Gamma$. In \cite{KrII89}, Kronheimer also proved a Torelli-type theorem classifying hyperk\"ahler ALE surfaces. In the $A_k$ case, these metrics were previously discovered by Eguchi-Hanson for $k=1$ \cite{EguchiHanson}, and 
by Gibbons-Hawking for all $k \geq 1$ \cite{GibbonsHawking}. 

\vspace{2mm}
\item Cyclic case: For the $\frac{1}{p} (1,q)$-action, Calderbank-Singer constructed a family of scalar-flat K\"ahler ALE metrics on the minimal resolution of any cyclic quotient singularity \cite{CalderbankSinger}. These metrics are toric and come in families of dimension $k-1$, where $k$ is the length of the corresponding Hirzebruch-Jung algorithm.  For $q=1$ and $q=p-1$, these metrics are the LeBrun negative mass metrics
and the toric multi-Eguchi-Hanson metrics, respectively \cite{LeBrunnegative, GibbonsHawking}.

\vspace{2mm}
\item Non-cyclic non-${\rm{SU}}(2)$ case: The existence of scalar-flat K\"ahler metrics on the minimal resolution of $\CC^2/\Gamma$, was shown by Lock-Viaclovsky \cite{LV14}.

\end{itemize}

 A natural question is whether the scalar-flat K\"ahler property is preserved under small deformations of complex structure.
In \cite{HVI},  we showed that for any scalar-flat K\"ahler ALE surface, 
all small deformations of complex structure admit scalar-flat 
K\"ahler ALE metrics, and so do all small deformations of the K\"ahler class. An informal statement is the following. 
\begin{theorem}
Let $(X,g,J)$ be a scalar-flat K\"ahler ALE surface. 
Then there is a finite-dimensional family $\mathfrak{F}$ of scalar-flat K\"ahler ALE metrics near $g$,  parametrized by a small ball in $\RR^d$, for some integer $d$.
This family $\mathfrak{F}$ is  ``versal'' in the following sense: it contains all possible scalar-flat K\"ahler ALE metrics ``near'' to the given scalar-flat K\"ahler ALE metric, up to diffeomorphisms which are sufficiently close to the identity. 
\end{theorem}
  A more precise statement of this theorem can be found in Section \ref{construction} below. The family $\mathfrak{F}$  is not ``universal'' since it is possible that $2$ metrics in $\mathfrak{F}$ could be isometric. However,  the orbit space of the group of biholomorphic isometries does give a universal moduli space, an informal statement of which is the following. 
\begin{theorem} 
The group  $\mathfrak{G}$ of holomorphic isometries of $(X,g,J)$ acts on $\mathfrak{F}$, and  each orbit represents a unique isometry class
of metric up to the action of diffeomorphisms which are sufficiently close the the identity. 
\end{theorem}
Again,  a more precise statement can be found in Section \ref{construction} below. As a consequence, the quotient space $\mathfrak{M} = \mathfrak{F}/\mathfrak{G}$ is the ``local moduli space
of scalar-flat K\"ahler ALE metrics near $g$.''
The local moduli space $\mathfrak{M}$ is not a manifold in general, but
its dimension is in fact well-defined, and we define 
$m = \dim ( \mathfrak{M})$.

\subsection{Deformations of the minimal resolution}

As mentioned above, there are families of examples of scalar-flat K\"ahler ALE metrics on minimal resolutions of isolated quotient singularities. We next recall
the definition of a minimal resolution.
\begin{definition}
\label{minresdef}
{\em Let $\Gamma \subset \U(2)$ be as above.  A smooth complex surface ${X}$ is called a {\textit{minimal resolution}} 
of $\CC^2/\Gamma$ if there is a holomorphic mapping $\pi : {X}\rightarrow \CC^2/\Gamma$ such 
that 
the restriction $\pi:{X} \setminus \pi^{-1}(0)\rightarrow \CC^2/\Gamma\setminus\{0\}$ is a biholomorphism, and
the set $\pi^{-1}(0)$ is a divisor in ${X}$ containing no $-1$ curves.
}
\end{definition}
The divisor $\pi^{-1}(0)$ is called the {\textit{exceptional divisor}} of the resolution.
In the cyclic case, the exceptional divisor is a string of rational curves with normal crossing singularities, and these are known are Hirzebruch-Jung strings. In the case that $\Gamma$ is non-cyclic, the exceptional divisor is a tree of 
rational curves with normal crossing singularities \cite{Brieskorn}. 
There are three Hirzebruch-Jung strings attached to a single curve,
called the {\textit{central rational curve}}. 
The self-intersection number of this curve will be denoted $-b_{\Gamma}$, and 
the total number of rational curves will be denoted by~$k_{\Gamma}$.

In the special case of a minimal resolution, our main result can be stated as follows. 
\begin{theorem}
\label{generalthm}
Let $(X,g,J)$ be any scalar-flat K\"ahler ALE metric on the minimal resolution of $\CC^2/\Gamma$, where $\Gamma \subset {\rm{U}}(2)$ is as above. Define 
\begin{align}
j_{\Gamma} =  2 \sum_{i = 1}^{k_{\Gamma}} (e_i -1),
\end{align}
where $-e_i$ is the self-intersection number of the $i$th rational curve, and $k_{\Gamma}$ is the number of rational curves in the exceptional divisor, and let 
\begin{align}
d_{\Gamma}=  j_{\Gamma} + k_{\Gamma}.
\end{align}
Then there is a family, $\mathfrak{F}$, parametrized by a ball in $\RR^{d_{\Gamma}}$, 
of scalar-flat K\"ahler metrics near $g$ which is ``versal''. 
The group  $\mathfrak{G}$ of holomorphic isometries of $(X,g,J)$
acts on $\mathfrak{F}$, and
the dimension $m_{\Gamma}$ of the local moduli space $\mathfrak{M} = \mathfrak{F}/\mathfrak{G}$  is given in Table \ref{dimtable}, where $e_{\Gamma}$ is the embedding dimension of $\CC^2/ \Gamma$. 
\begin{table}[h]
\caption{Dimension of local moduli space of scalar-flat K\"ahler metrics}
\label{dimtable}
\begin{tabular}{ll l l}
\hline
$\Gamma\subset{\rm U}(2)$ & $d_{\Gamma}$ & $m_{\Gamma}$\\
\hline\hline
\vspace{1mm}
$ \frac{1}{3}(1,1)$ & $5$ & $2$\\
\vspace{1mm}
$ \frac{1}{p}(1,1) , p \geq 4$ &  $ 2p -1   $ & $2 p - 5$\\
\vspace{1mm}
$ \frac{1}{p}(1,q), q \neq 1, p-1 $ & $j_{\Gamma} + k_{\Gamma}$ & $2 e_{\Gamma} + 3k - 8$\\
non-cyclic, not in ${\rm{SU}}(2)$ & $j_{\Gamma} + k_{\Gamma}$  & $2 e_{\Gamma} + 3k - 7 $ \\
\hline
\end{tabular}
\end{table}
\end{theorem}
A description of the possible groups $\Gamma$ and other explicit formulas for $m_{\Gamma}$ 
can be found in Section \ref{exsec} below. 
\begin{remark}
\label{hrem}
{\em We did not include the ${\rm{SU}}(2)$ case in the above since 
the dimension of the moduli space of hyperk\"ahler metrics 
is known to be $3 k - 3$ in the $A_k, D_k$ and $E_k$ cases for $k \geq 2$, 
and equal to $1$ in the $A_1$ case \cite{KrCR}. Our method of
parametrizing by complex structures and K\"ahler classes overcounts 
in this case, since a hyperk\"ahler metric is K\"ahler with respect to a
$2$-sphere of complex structures, see Section \ref{exsec} for some further remarks. 
For other related results in the Ricci-flat case, see \cite{ConlonHeinII, Suvaina_ALE}. 
}
\end{remark}

It turns out that the moduli count in Theorem \ref{generalthm} is correct not just for the minimal resolution, but for any generic scalar-flat K\"ahler ALE surface which can be continuously deformed to the minimal resolution. 
\begin{theorem}[\cite{HVI}]
\label{t1.8}
Let $(X,g,J)$ be any scalar-flat K\"ahler ALE surface which deforms to 
the minimal resolution of $\CC^2/\Gamma$ through a path $(X,g_t,J_t)$ $(0\leq t\leq 1)$, 
where $g_1 = g$, $g_0$ is the minimal resolution, 
and $\|g_t-g_s\|_{C^{k,\alpha}_\delta(g_0)}\leq C\cdot |s-t|$ with $C>0$ a uniform constant for
any $0\leq s,t\leq 1$, $k\geq 4$, $-2<\delta<-1$.
If $\mathfrak{G}(g)  = \{e\}$ then the local moduli space 
$\mathfrak{F}$ is smooth near $g$ and is a manifold of dimension $m = m_{\Gamma}$.
\end{theorem}
The proof of this theorem is more or less a direct application of Theorem~\ref{generalthm} together with the basic fact that the index of a strongly continuous family of Fredholm operators is constant.
\begin{remark}{\em
It was recently shown that K\"ahler ALE surface with group $\Gamma\subset {\rm{U}}(2)$ is birational to a deformation of $\CC^2/\Gamma$ \cite{HRS16}. There are several possible components of the deformation of such a cone, so the above result gives the dimension of the moduli space for the ``Artin component'' of deformations of $\CC^2/\Gamma$, which is the component with maximal dimension. 
}
\end{remark}

\subsection{Acknowledgements}  
This article is dedicated to Gang Tian on the occasion of the 60th birthday. The second author is extremely grateful to Tian for his collaboration, friendship, and generosity throughout the years since we first met over 20 years ago. 

\section{Construction of the local moduli space}
\label{construction}
In this section, we will give a survey of the main results in \cite{HVI}. 
We  first recall some basic facts regarding deformations of complex structures. 
For a complex manifold $(X,J)$, let $\Lambda^{p,q}$ denote the bundle of $(p,q)$-forms, and let $\Theta$ denote the holomorphic tangent bundle. The deformation complex 
corresponds to a real complex as shown in the commutative diagram 
\begin{equation}
\begin{tikzcd}[column sep=2.1cm]
\label{cd1}
\Gamma(\Theta) \arrow[r, "\overline{\partial}"] \arrow[d, "Re"] & \Gamma( \Lambda^{0,1} \otimes \Theta) \arrow[r, "{\overline{\partial}}"] \arrow[d, "Re"] & \Gamma(\Lambda^{0,2} \otimes \Theta) \arrow[d, "Re"]\\
\Gamma(TX) \arrow[r,"Z \mapsto -\frac{1}{2} J \circ \mathcal{L}_{Z}J"]  & \Gamma(End_{a}(TX)) 
\arrow[r, "I \mapsto \frac{1}{4} J \circ N_J'(I)"] &  \Gamma\big( \{\Lambda^{0,2} \otimes \Theta \oplus \Lambda^{2,0} \otimes \overline{\Theta} \}_{\RR} \big),
\end{tikzcd}
\end{equation}
where $\mathfrak{L}_Z J$ is the Lie derivative of $J$,
\begin{align}
End_{a}(TX) = \{I\in End(TX): IJ = -JI\},
\end{align}
and $N_J'$ is the linearization of Nijenhuis tensor 
\begin{align}
N(X,Y) = 2\{[JX,JY]-[X,Y]-J[X,JY]-J[JX,Y]\}
\end{align}
at $J$. 
Each isomorphism $Re$ is simply taking the real part of a section. 
If $g$ is a Hermitian metric compatible with $J$, then let $\square$ denote the $\bar\partial$-Laplacian 
\begin{align}
\square  \equiv \bar\partial^*\bar\partial+\bar\partial\bar\partial^*,
\end{align}
where $\bar\partial^*$ denotes the formal $L^2$-adjoint. 
Each complex bundle in the diagram \eqref{cd1} admits a $\square$-Laplacian,
and these correspond to real Laplacians on each real bundle 
in \eqref{cd1}. We will use the same $\square$-notation for 
these real Laplacians. 

We next define the spaces of harmonic sections which will appear in the statement of the main result. 
\begin{definition}
{\em
Let $(X,g,J)$ be a K\"ahler ALE surface. For any bundle $E$ in 
the diagram \eqref{cd1}, and $\tau \in \RR$, define
\begin{align}
\H_{\tau}(X,E) &= \{\theta\in \Gamma(X,E): \square\theta=0, \theta=O(r^{\tau}) 
 \mbox{ as } r \rightarrow \infty \}.
\end{align}
Define 
\begin{align}
\WW = \{ Z \in \mathcal{H}_1(X, TX) \ | \ \mathfrak{L}_Z g  = O (r^{-1}), \ \mathfrak{L}_Z J = O(r^{-3}),\mbox{ as } r \rightarrow \infty \}.
\end{align}
Finally, define the real subspace
\begin{align}
\label{hess}
\H_{ess} (X,End_a(TX)) \subset  \H_{-3}(X,End_a(X)) 
\end{align}
to be the $L^2$-orthogonal complement in $\H_{-3}(X,End_a(X))$
of the subspace 
\begin{align}
\VV = \{ \theta \in \H_{-3}(X,End_a(TX)) \  | \ 
\theta =  J\circ \mathfrak{L}_Z J, \  Z \in \WW \}.
\end{align}
}
\end{definition}
The subscript {\it{ess}} in \eqref{hess} is short for essential, and 
is necessary because there is a gauge freedom of Euclidean 
motions in the definition of ALE coordinates, so that 
element of $\VV$ are not really essential deformations, i.e., they can be gauged away. 

To state the main result precisely, we need to define weighted H\"older spaces.
\begin{definition}{\em
Let $E$ be a tensor bundle on $X$, with Hermitian metric $\Vert \cdot\Vert_h$. Let $\varphi$ be a smooth section of $E$. We fix a point $p_0\in X$, and define $r(p)$ to be the distance between $p_0$ and $p$. Then define 
\begin{align}
\Vert \varphi\Vert_{C^{0}_\delta} &:= \sup_{p\in X}\Big\{\Vert\varphi(p)\Vert_h\cdot (1+r(p))^{-\delta}\Big\}\\
\Vert \varphi\Vert_{C^{k}_\delta} &:= \sum_{|\mathcal{I}|\leq k}\sup_{p\in X}
\Big\{\Vert\nabla^{\mathcal{I}} \varphi(p)\Vert_h\cdot (1+r(p))^{-\delta+|\mathcal{I}|}\Big\},
\end{align}
where $\mathcal{I} = (i_1,\ldots,i_n),|\mathcal{I}|=\sum_{j=1}^n i_j$.
Next, define
\begin{align}
[\varphi]_{C^{\alpha}_{\delta-\alpha}} &:= \sup_{0<d(x,y)<\rho_{inj}}\Big\{\min\{r(x),r(y)\}^{-\delta+\alpha}\frac{\Vert\varphi(x)-\varphi(y)\Vert_h}{d(x,y)^\alpha}\Big\},
\end{align}
where $0<\alpha<1$, $\rho_{inj}$ is the injectivity radius, and $d(x,y)$ is the distance between $x$ and $y$. The meaning of the tensor norm is to use parallel transport along the unique minimal geodesic from  $y$ to $x$, and then take the norm of the difference 
at $x$. 
The weighted H\"older norm is defined by 
\begin{align}
\Vert\varphi\Vert_{C^{k,\alpha}_\delta} &:= \Vert \varphi\Vert_{C^{k}_\delta}+\sum_{|\mathcal{I}|=k}[\nabla^{\mathcal{I}} \varphi]_{C^{\alpha}_{\delta-k-\alpha}},
\end{align}
and the space $C^{k,\alpha}_{\delta}(X, E)$ is the closure of $\{\varphi \in C^{\infty}(X,E): \Vert \varphi\Vert_{C^{k,\alpha}_\delta}<\infty\}$.
}
\end{definition}

The main result of \cite{HVI} is the following.
\begin{theorem}[\cite{HVI}]
\label{tmain}
Let $(X,g,J)$ be a scalar-flat K\"ahler ALE surface.
Let $-2 < \delta < -1$, $0 < \alpha < 1$,  and $k$ an integer with $k \geq 4$
 be fixed constants. 
Let $B^1_{\epsilon_1}$ denote an $\epsilon_1$-ball in $\H_{ess} ( X,End_a(TX)) $, 
$B^2_{\epsilon_2}$ denote an $\epsilon_2$-ball in $\H_{-3}(X,\Lambda^{1,1})$ (both using the $L^2$-norm). 
Then there exists $\epsilon_1 >0$ and $\epsilon_2 >0$ and a family $\mathfrak{F}$ of scalar-flat K\"ahler metrics near $g$, parametrized by 
$B^1_{\epsilon_1} \times B^2_{\epsilon_2}$, that is, there is a differentiable mapping 
\begin{align}
F : B^1_{\epsilon_1} \times B^2_{\epsilon_2} \rightarrow {Met}(X), 
\end{align}
into the space of smooth Riemannian metrics on $X$, 
with $\mathfrak{F} = F(B^1_{\epsilon_1} \times B^2_{\epsilon_2})$ satisfying the following ``versal'' property:
there exists a constant $\epsilon_3 >0$ 
such that for any scalar-flat K\"ahler metric $\tilde{g}\in B_{\epsilon_3}(g)$,
there exists a diffeomorphism $\Phi: X \rightarrow X$, $\Phi \in C^{k+1,\alpha}_{loc}$,
such that $\Phi^* \tilde{g} \in \mathfrak{F}$, 
where 
\begin{align}
\label{wnd}
B_{\epsilon_3}(g) =  \{g' \in C^{k,\alpha}_{loc}(S^2(T^*X))  \ | \ \Vert g - g' \Vert_{C^{k,\alpha}_{\delta}(S^2(T^*X))} < \epsilon_3 \}.
\end{align} 
\end{theorem}

\subsection{Outline of Proof of Theorem \ref{tmain}}
The main steps in the proof of Theorem~\ref{tmain} are the following. 

\vspace{2mm}
\noindent
Step I:  One first analyzes deformations of complex structures using an adaptation of Kuranishi's theory  \cite{Kuranishi},
to ALE spaces. To first order, the almost complex structures near a given ALE K\"ahler metric are in correspondence with sections in $\Gamma(\Lambda^{0,1} \otimes \Theta)$. 
The integrable complex structures solve a nonlinear elliptic equation, modulo diffeomorphisms. By imposing a divergence-free gauging condition, we obtain a finite-dimensional Kuranishi family which is parametrized by decaying harmonic sections in  $\mathcal{H}_{-3}(X, \Lambda^{0,1} \otimes \Theta)$. Unobstructedness follows from a vanishing theorem, which relies on
some analysis of the complex analytic compactifications of K\"ahler ALE spaces, due to Hein-LeBrun-Maskit \cite{HL15, LeBrunMaskit}.
An important point is that since the manifold is non-compact, the sheaf cohomology group $H^1(X, \Theta)$, which vanishes in the Stein case,  should be replaced by an appropriate space of decaying harmonic forms.

\vspace{2mm}
\noindent
Step II: Several key results about
gauging and diffeomorphisms are needed to prove ``versality'' of the 
family constructed. Our main infinitesimal slicing result is the following. 
\begin{lemma}
\label{dfree}
Let $(X,J_0,g_0)$ be a K\"ahler ALE surface with $J_0,g_0 \in C^{\infty}$. There exists an $\epsilon_1'>0$ such that for any complex structure $\Vert J_1-J_0\Vert_{C^{k,\alpha}_{\delta}}<\epsilon_1'$, where $k\geq 3, \alpha\in (0,1),\delta\in (-2,-1)$, there exists a unique diffeomorphism $\Phi$,
of the form $\Phi_Y$ (see \eqref{phidef} below) for $Y \in C^{k+1,\alpha}_{\delta+1}(TX)$
such that $\Phi_Y^*(J_1)$ is in the divergence-free gauged Kuranishi family. 
\end{lemma}
Essentially, this shows that the divergence-free gauge gives a local slice transverse to the ``small'' diffeomorphism group action.
However, a more refined gauging procedure is needed in order to construct the Kuranishi family of ``essential'' deformations. As stated above, this refined gauging is necessary because of the freedom of Euclidean motions in the definition of an ALE metric, which means that there are decaying elements in the kernel of the linearized operator which can be written as Lie derivatives of linearly growing vector fields. These directions are not true moduli directions, and we show that they can be ignored modulo diffeomorphisms. Thus we can restrict attention to the subspace of essential deformations defined in \eqref{hess} above. 

\vspace{2mm}
\noindent
Step III: Next, one needs to generalize Kodaira-Spencer's stability theorem for K\"ahler structures \cite{KSIII} to the ALE setting, to prove that the above deformations retain the ALE K\"ahler property. This was proved using some arguments similar to that of Biquard-Rollin \cite{BiquardRollin}.

\vspace{2mm}
\noindent
Step IV: To study the deformations of the scalar-flat K\"ahler structure, we then adapted the LeBrun-Singer-Simanca theory of deformations of extremal K\"ahler metrics to the ALE setting \cite{LS93, LeBrun_Simanca}. 
Denote $S(\omega_0+\sqrt{-1}\partial\bar\partial f)$ as the scalar curvature of $X$ with metric $\omega_0+\sqrt{-1}\partial\bar\partial f$. We consider $S$ as mapping between weighted H\"older spaces,
\begin{align}
\begin{split}
S : C^{k,\alpha}_{\epsilon}(X)&\rightarrow C^{k-4,\alpha}_{\epsilon-4}(X)\\
 f &\mapsto S(\omega_0+\sqrt{-1}\partial\bar\partial f).
\end{split}
\end{align}
If $\omega_0$ is scalar-flat, the linearized operator is $L(f) = - (\bar\partial\bar\partial^{\#})^*(\bar\partial\bar\partial^{\#})(f)$, where the operator $\bar\partial^{\#}f = g_0^{i,\bar{j}}\bar\partial_j f$. We showed that the linearized map is surjective for $0 < \epsilon < 1$,
and then an application of the implicit function theorem completes the proof. 
\subsection{Universality}
As mentioned above, the family $\mathfrak{F}$ is not necessarily ``universal'',
because some elements in $\mathfrak{F}$ might be isometric.
To construct a universal moduli space, we need to describe a neighborhood of the identity in the space of diffeomorphisms. If $(X,g)$ is an ALE metric, and $Y$ is a vector field on $X$, 
the Riemannian exponential mapping $\exp_p : T_pX 
\rightarrow X$ induces a mapping 
\begin{align}
\Phi_Y : X \rightarrow X
\end{align}
by 
\begin{align}
\label{phidef}
\Phi_Y (p) = \exp_p (Y).
\end{align}
If $Y \in C^{k,\alpha}_{s}(TX)$ has sufficiently small norm, 
($s < 0$ and $k$ will be determined in specific cases)
then $\Phi_Y$ is a diffeomorphism. We will use the correspondence 
$Y \mapsto \Phi_Y$ to parametrize a neighborhood of the identity, 
analogous to  \cite{Biquard06}.
\begin{definition}
\label{smalldef}
{\em 
We say that $\Phi: X \rightarrow X$  
is a {\it{small diffeomorphism}} if $\Phi$ is of the form 
$\Phi = \Phi_Y$ for some vector field $Y$ satisfying
\begin{align}
\Vert Y \Vert_{C^{k+1,\alpha}_{\delta+1}} < \epsilon_4
\end{align}
for some $\epsilon_4 >0$ sufficiently small which depends on $\epsilon_3$.
}
\end{definition}
The following result shows that after taking a quotient 
by an action of the holomorphic isometries of the central fiber $(X,g,J)$, the family $\mathfrak{F}$ is in fact universal (up to small diffeomorphisms).
\begin{theorem}[\cite{HVI}]
\label{t2}
Let $(X,g,J)$ be as in Theorem \ref{tmain},
and let $\mathfrak{G}$ denote the group of holomorphic isometries of $(X,g,J)$. 
Then there is an action of $\mathfrak{G}$ on $\mathfrak{F}$ with the following properties.
\begin{itemize}

\item Two metrics in $\mathfrak{F}$ are isometric 
if they are in the same orbit of $\mathfrak{G}$. 

\item If two metrics in $\mathfrak{F}$ are isometric by a small
diffeomorphism then they must be the same.  
\end{itemize}
\end{theorem}
This theorem was proved in \cite{HVI} more or less by keeping track of the action of $\mathfrak{G}$ in every step of the proof of Theorem \ref{tmain}. 
Since each orbit represents a unique isometry class
of metric (up to small diffeomorphism), we will refer to the quotient $\mathfrak{M} = \mathfrak{F}/\mathfrak{G}$ as the ``local moduli space
of scalar-flat K\"ahler ALE metrics near $g$.''
The local moduli space $\mathfrak{M}$ is not a manifold in general, but
since $\mathfrak{F}$ is of finite dimension, and $\mathfrak{G}$ is a compact group action on $\mathfrak{F}$, the dimension~$m$ of $\mathfrak{M} = \mathfrak{F}/\mathfrak{G}$ is well-defined.
In the non-hyperk\"ahler case, 
\begin{align}
m = d- ( \text{the dimension of a maximal orbit of } \mathfrak{G}),
\end{align}
where 
\begin{align}
d = \dim_{\RR} \big( \H_{ess} ( X,End_a(TX)) \big) 
+ b_2 (X),
\end{align}
where $b_2(X)$ is the second Betti number of $X$.
(For the hyperk\"ahler case, recall Remark~\ref{hrem}.) 
\begin{remark}{\em
We note that the local moduli space of metrics contains small rescalings, i.e, $g\mapsto \frac{1}{c^2}g(c\cdot,c\cdot)$ for $c$ close to $1$. If one considers scaled 
metrics as equivalent (which we do not), then the dimension would decrease by $1$.
}
\end{remark}

\section{The case of the minimal resolution}
\label{defminsec}
Let $X$ denote the minimal resolution of $\CC^2/\Gamma$, where $\Gamma$ is a finite subgroup of ${\rm{U}}(2)$ without complex reflections. 
The divisor $E = \cup_i E_i$ is a union of irreducible components which are
rational curves, with only normal crossing singularities.
Let $Der_{E}({X})$ denote the sheaf dual to logarithmic $1$-forms along $E$ (see \cite{Kawamata}).  We note that $Der_{E}({X})$ is a locally free sheaf of rank $2$, see \cite{Wahl1975}.
Away from $E$, this is clear. If $p\in E_i$, we can choose a holomorphic coordinate chart $\{z_1, z_2\}$ such that near $p$, $E_i = \{z_1=0\}$. Then local sections 
of $Der_{E}({X})$ are generated by $\{z_1\frac{\partial}{\partial z_1},\frac{\partial}{\partial z_2}\}$. 

Since $E$ is composed of rational curves whose self-intersection numbers are negative, we have $H^0(E,\oo_E(E))=0$. The short exact sequence
\begin{align}
0\rightarrow Der_E({X}) \rightarrow \Theta_{{X}} \rightarrow \oo_E(E)\rightarrow 0,
\end{align}
then induces an exact sequence of cohomologies
\begin{align}
\label{es}
0\rightarrow H^1(X,Der_E(X))\rightarrow H^1(X,\Theta)\rightarrow H^1(E,\oo_E(E))\rightarrow H^2(X,Der_E(X)).
\end{align}
By Siu's vanishing theorem (\cite{Siu}), since $X$ is a non-compact $\sigma$-compact complex manifold, for any coherent analytic sheaf $\mathscr{F}$ on $X$, the top degree sheaf cohomology $H^2(X,\mathscr{F})$ is trivial. Consequently, $H^2(X,Der_E(X))=0$,

 In \cite{HVI} we cited several papers from algebraic geometry  \cite{BKR, Brieskorn, Laufer1973, Wahl1975}, to conclude that  $H^1(X,Der_E(X))=0$. 
In this section, we will give a different proof of the following result, using some tools from geometric analysis.
\begin{theorem}
\label{dimh1t}
For $X$ the minimal resolution of $\CC^2/\Gamma$, we have
\begin{align}
\label{dimh1}
\dim_{\CC}(H^1(X,\Theta)) = \sum_{j=1}^{k_\Gamma}(e_j-1).
\end{align}
\end{theorem}
In relation to the construction of the moduli space of scalar-flat K\"ahler ALE metrics,  we need to construct a weighted version of Hodge theory, that links the sheaf cohomology with the decaying harmonic forms. Recall that in Theorem \ref{tmain}, the deformations of 
complex structure are parametrized by decaying harmonic sections 
in $\H_{-3}(X,\Lambda^{0,1}\otimes\Theta)$ which are ``essential'', that is, they are in the subspace $\VV$.  But, in the case of the minimal resolution, it turns out the dimension 
of this space is equal to the dimension of $H^1(X,\Theta)$.
\begin{theorem}[\cite{HVI}]
\label{t2.3}
Let $(X,g,J)$ denote the minimal resolution of $\CC^2/\Gamma$ with any ALE K\"ahler metric $g$ of order $\tau > 1$. Then
\begin{align}
H^1(X,\Theta) \cong  \H_{-3}(X,\Lambda^{0,1}\otimes\Theta)
\cong \H_{ess}(X,\Lambda^{0,1}\otimes\Theta)
\end{align}
\end{theorem}
As a consequence, the dimension of the space of essential deformations is given by \eqref{dimh1}.  This is very special to the case of the minimal resolution. In the Stein case, Theorem \ref{t2.3} is not true in general because the sheaf cohomology group necessarily vanishes.
\subsection{Cyclic quotient singularity}
We will first prove Theorem \ref{dimh1t} in the case of a cyclic group, using a direct arguments involving only sheaf theory. 
Consider a cyclic quotient singularity of the form 
$\Gamma = \frac{1}{p}(1,q)$ ($p\geq q$). 
We will first give some additional detail regarding the Hirzebruch-Jung resolutions. Details can be found in \cite{Re, Ko07}.

The continued fraction described below in formula \eqref{fraction},  
can also be represented by lattice points
\begin{align}
c_0 = (1,0), c_1=\frac{1}{p}(1,q),\ldots, c_{m+1}=(0,1),
\end{align}
with iterative relation 
\begin{align}
\left( \begin{matrix}c_i \\ c_{i+1}\end{matrix}\right) = \left(\begin{matrix}0 && 1\\ -1 && e_i\end{matrix}\right)\cdot \left(\begin{matrix}c_{i-1} \\ c_i\end{matrix}\right).
\end{align}

Meanwhile, the dual continued fraction $\frac{p}{p-q} = [a_1,\ldots,a_k]$ can be used to give the invariant polynomials: 
\begin{align}
\label{invpoly}
u_0 = x^p, u_1=x^{p-q}y,u_2,\ldots,u_k,u_{k+1}=y^p,
\end{align}
which satisfy the relation $u_{i-1}u_{i+1}=u_i^{a_i}$.

The polynomials $\{u_0,\ldots,u_{k+1}\}$ give an embedding of the cone in $\CC^{k+2}$.
Let 
\begin{align}
c_0 = (s_0,t_0), \ldots, c_{m+1} = (s_{m+1}, t_{m+1})
\end{align}
be lattice points, where $s_0 = 0, t_0 = 1, s_m = 1, t_m=0, s_{i+1}>s_i, t_{i+1}<t_i$.
Let $\{\eta_i,\xi_i\}$ ($0\leq i\leq m+1$) be monomials forming the dual basis 
to $\{c_i,c_{i+1}\}$, i.e., 
\begin{align}
c_i(\eta_i)=1, c_i(\xi_i)=0, c_{i+1}(\eta_i)=0, c_{i+1}(\xi_i)=1.
\end{align}
\begin{proposition}
The numbers $s_i, t_i$ satisfy the relation
\begin{align} 
\label{recrel}
t_is_{i+1}-t_{i+1}s_i = \frac{1}{p}.
\end{align}
\end{proposition}
\begin{proof}
We prove it by induction. First, note that $c_0 = (0,1), c_1 = \frac{1}{p}(1,q)$. 
Then $t_0s_1-t_1s_0 = \frac{1}{p}$. Next, assume that $t_{i-1}s_i-t_is_{i-1}=\frac{1}{p}$. 
By the recursive formula 
 $c_{i+1}+c_{i-1} = a_ic_i$, it follows that 
\begin{align}
(s_{i+1},t_{i+1})+(s_{i-1},t_{i-1}) = a_i(s_i,t_i).
\end{align}
Then we have
\begin{align}
 s_{i+1} = a_is_i-s_{i-1},  \ t_{i+1} = a_it_i-t_{i-1}. 
\end{align}
So finally, 
\begin{align} 
t_is_{i+1}-t_{i+1}s_i = t_i(a_is_i-s_{i-1})-(a_it_i-t_{i-1})s_i = t_{i-1}s_i-t_is_{i-1} = \frac{1}{p}.
\end{align}
\end{proof}
By the formula \eqref{recrel}, we have that $\eta_i = p\cdot (-t_{i+1},s_{i+1}), \xi_i = p\cdot (t_i,-s_i)$. Then 
\begin{align}
\label{cover}
\xi_i = \frac{x^{pt_i}}{y^{ps_i}}, \eta_i = \frac{y^{ps_{i+1}}}{x^{pt_{i+1}}}, \xi_{i+1} = \frac{x^{pt_{i+1}}}{y^{ps_{i+1}}}, \eta_{i+1} = \frac{y^{ps_{i+2}}}{x^{pt_{i+2}}}.
\end{align}
It follows that the coordinate transition from  
$\{ \eta_i,\xi_i \}$ to $\{ \eta_{i+1},\xi_{i+1} \}$ for 
$\xi_{i+1}\neq 0$, is given by 
\begin{align}
\label{ct}
\eta_i = \xi_{i+1}^{-1},\ \eta_{i+1} = \eta_i^{e_{i+1}}\xi_i, \; (0\leq i\leq m-1)
\end{align}
which defines an acyclic cover $Y = Y_0\cup Y_1\ldots \cup Y_m$ of $X$ satisfying  
\begin{align}
Y_i\cap Y_{i+1} \simeq \CC\times \CC^*, \ Y_i\cap Y_{i+k} = Y_i\cap Y_{i+1}\ldots\cap Y_{i+k},
\end{align}
see \cite[Theorem 3.2]{Re}.
For use below, we record the following formulae:
\begin{align}
\label{Der}
\begin{split}
\frac{\partial}{\partial \eta_i} &= \frac{1}{\eta_i}(s_i x\frac{\partial}{\partial x} + t_i y\frac{\partial}{\partial y})\\
\frac{\partial}{\partial \xi_i} &= \frac{1}{\xi_i}(s_{i+1}x\frac{\partial}{\partial x}+t_{i+1} y\frac{\partial}{\partial y}) \\
\frac{\partial}{\partial x} &= \frac{p}{x}(t_i\xi_i\frac{\partial}{\partial\xi_i}-t_{i+1}\eta_i\frac{\partial}{\partial\eta_i})\\
\frac{\partial}{\partial y} &=\frac{p}{y}(-s_i\xi_i\frac{\partial}{\partial\xi_i}+s_{i+1}\eta_i\frac{\partial}{\partial\eta_i}).
\end{split}
\end{align}
With these preliminaries, we can now prove the following result. 
\begin{lemma}
\label{cyclic}
When $\Gamma$ is cyclic, $H^1(X,Der_E(X)) = 0$ for the minimal resolution $X$ of $\CC^2/\Gamma$.
\end{lemma}
\begin{proof}
From \eqref{ct} above, 
\begin{align}
\frac{\partial}{\partial \xi_{i+1}} &= -\eta_i^2\frac{\partial}{\partial \eta_i}+e_{i+1}\eta_i\xi_i\frac{\partial}{\partial \xi_i},\\
\frac{\partial}{\partial \eta_{i+1}} &= \eta_i^{-e_{i+1}}\frac{\partial}{\partial \xi_i}.
\end{align}
The sections of $Der_E(X)$ are generated by 
\begin{align}
\Big\{ \frac{\partial}{\partial\eta_i},\xi_i\frac{\partial}{\partial\xi_i} \Big\} ,  \ 
\Big\{ \frac{\partial}{\partial\xi_{i+1}},\eta_{i+1}\frac{\partial}{\partial\eta_{i+1}} \Big\}
\end{align}
on $Y_i,Y_{i+1}$ respectively.
For $\theta_i\in \Gamma(Y_i,Der_E (X))$, $\theta_i$ can be expanded as a Laurent series:
\begin{align} 
\theta_i = \sum_{k\geq 0,l\geq 0} a^i_{k,l} \eta_i^k\xi_i^l \frac{\partial}{\partial \eta_i}+b^i_{k,l}\eta_i^k\xi_i^{l+1}\frac{\partial}{\partial \xi_i}.
\end{align} 
For $\theta_{i+1}\in \Gamma(Y_{i+1},Der_E(X))$,
\begin{align}
\theta_{i+1} = \sum_{k\geq 0,l\geq 0}a^{i+1}_{k,l}\xi_{i+1}^k\eta_{i+1}^l\frac{\partial}{\partial \xi_{i+1}}+b^{i+1}_{k,l}\xi_{i+1}^k\eta_{i+1}^{l+1}\frac{\partial}{\partial\eta_{i+1}}.
\end{align}
For $\theta_{i,i+1}\in \Gamma(Y_i\cap Y_{i+1},Der_E(X))$ on the intersection $Y_i\cap Y_{i+1}$ where $\eta_i\neq 0$,
\begin{align}
\theta_{i,i+1} = \sum_{k\in\ZZ,l\geq 0} a^{i,i+1}_{k,l} \eta_i^k \xi_i^l\frac{\partial}{\partial\eta_i}+b^{i,i+1}_{k,l}\eta_i^k\xi_i^{l+1}\frac{\partial}{\partial\xi_i}.
\end{align}
By the transition formula \eqref{ct}, 
\begin{align}
\begin{split}
\theta_{i+1} = \sum_{k\geq 0,l\geq 0}\Big\{ &-a^{i+1}_{k,l}\eta_i^{-k+le_{i+1}+2}\xi_i^l\frac{\partial}{\partial \eta_i}\\
&+(a^{i+1}_{k,l}e_{i+1}\eta_i^{-k+le_{i+1}+1}\xi_i^{l+1}+b^{i+1}_{k,l}\eta_i^{-k+le_{i+1}}\xi_i^{l+1})\frac{\partial}{\partial \xi_i} \Big\},
\end{split}
\end{align}
which shows that the exponents of $\eta_i$ in $\theta_{i+1}$ can be any negative integers. Then it is clear that for any $\theta_{i,i+1}$, there exist $\theta_i,\theta_{i+1}$ such that on $Y_i\cap Y_{i+1}$, $\theta_{i,i+1} = \theta_{i+1}-\theta_i$. Furthermore, if $\{\theta_{k,l}\}$ $(k<l)$ is closed, then $\theta_{k,l} = \theta_{k,l+1}+\ldots+\theta_{l-1,l}$. Then $\{\theta_{k,l}\}$ is determined if and only if the set of consecutive elements $\{\theta_{i,i+1}\}$ is determined. 
These arguments imply that any closed $\{\theta_{k,l}\}$ is exact, so $H^1(X,Der_E(X)) = 0$.
\end{proof}

\subsection{Non-cyclic quotient singularities}
In Lemma \ref{cyclic}, we have shown that when $\Gamma$ is cyclic, $H^1(X,Der_E(X)) = 0$, which implies that $H^1(X,\Theta) \simeq H^1(E,\oo_E(E))$. 
In the following, we will use a relative index theorem to show this also holds for the general case. 
\begin{proposition}
\label{counting}
Let $X\mapsto \CC^2/\Gamma$ be a minimal resolution, where $\Gamma\subset {\rm{U}}(2)$ with no complex reflections. Then $H^1(X,\Theta)\simeq H^1(E, \oo_E(E))$ and $H^1(X,Der_E(X))=0$.
\end{proposition}
\begin{proof}
Assume $\Gamma$ is non-cyclic. We will first construct
a K\"ahler form on
the minimal resolution $X$ of $\CC^2/\Gamma$ by gluing 
Calderbank-Singer ALE surfaces $Y_j$ ($j=1,2,3$) to a quotient of a LeBrun orbifold 
$X_0$, which has three cyclic quotient singularities on the central rational curve. 
The following is a sketch of the gluing procedure, details of a similar construction can be found in \cite{LV14}. Note that in that paper, this gluing was used to proced scalar-flat K\"ahler metrics, while in the following argument we will be considering a different operator $P$, see \eqref{Pdef} below.  

Let $x_i$, $i = 1,2,3$ be the cyclic quotient singularities of $X_0$ with 
group $\Gamma_i$. 
Let $(z_i^1,z_i^2)$ be local holomorphic coordinates on $U_i\setminus\{x_i\}$. 
Let $\omega_{X_0}$ be the K\"ahler form of the LeBrun metric. 
Then $\omega_{X_0}$ admits an expansion 
\begin{align}
\omega_{X_0} = \frac{\sqrt{-1}}{2}(\partial\bar\partial|z_i|^2+\partial\bar\partial\xi_i)
\end{align}
on $U_i\setminus\{x_i\}$, where $\xi_i$ is a potential function satisfying 
$\xi_i= O(|z_i|^4)$. 
For the LeBrun orbifold $X_0$, outside of a compact subset, it admits a holomorphic coordinate $(v_1,v_2)$. Let $Y_i$ denote the minimal resolution of $\CC^2/\Gamma_i$. 
Outside of a compact subset of $Y_i$, there exist holomorphic coordinates $(u^1_i,u^2_i)$.
Let $\omega_{Y_i}$ be a K\"ahler form on $Y_i$ corresponding to any Calderbank-Singer metric on $Y_i$. From \cite{RS09}, the K\"ahler form admits an expansion 
\begin{align}
\omega_{Y_i} = \frac{\sqrt{-1}}{2}(\partial\bar\partial|u_i|^2+\partial\bar\partial\eta_i),
\end{align}
where $\eta_i- c \log(|u|^2) = O(|u_i|^{-1})$, for some constant $c$. In fact, a similar expansion holds for any scalar-flat K\"ahler ALE surface on a resolution \cite{ALM_resolution}. 

Next, we construct a K\"ahler form on $X$. 
Choose two small positive numbers $a,b$, we glue the regions $\frac{1}{a}\leq |u_i|\leq \frac{4}{a}$ and $b\leq |z_i|\leq 4b$ , by letting $z_i = ab\cdot u_i$. This mapping is biholomorphic in the intersection. Let $\rho$ be a smooth cutoff function satisfying 
$\rho(t) = 1$ when $t\leq 1$, $\rho=0$ when $t\geq 2$.  Let 
\begin{align}
\begin{split}
\omega_b &= \begin{cases}\frac{\sqrt{-1}}{2}(\partial\bar\partial|z_i|^2+\partial\bar\partial((1-\rho(\frac{|z_i|}{2b}))\xi_i(z_i))) & \text{ if } |z_i|\leq b \\ \omega_{X_0} & \text{ if } |z_i|\geq 4b\end{cases} \\
\omega_a &= \begin{cases}\frac{\sqrt{-1}}{2}(\partial\bar\partial|u_i|^2+\partial\bar\partial(\rho(a|u_i|)\eta_i(u_i))) & \text{ if } |u_i|\geq 4a^{-1} \\ \omega_{Y_i} & \text{ if } |u_i|\leq a^{-1}\end{cases}.
\end{split}
\end{align}
Then we define
\begin{align}
\omega_{a,b} &= \begin{cases} a^{-2}b^{-2}\omega_{b} & \text{ if } |z_i|\geq 2b\\ \omega_a & \text{ if } |u_i|\leq 2a^{-1}\end{cases}.
\end{align} 
For $a, b$ sufficiently small, $\omega_{a,b}$  is a K\"ahler form on $X$.
Since $\omega_{X_0}$ was ALE of order $2$, the K\"ahler metric $\omega_{a,b}$ 
is also ALE of order $2$. 

Next, choose $R_1,R_2,R_3$, such that $0<2 R_1<R_2,R_3>0$, and define smooth functions $r_1,r_2,r_3$ as:
\begin{align}
\begin{split}
r_{1}(x) &= \begin{cases} |z_i| & \text{ if } |z_i|\leq R_1 \\ 1 & \text{ if } |z_i|\geq 2R_1\end{cases}\\
r_2(x) &= \begin{cases} 1 & \text{ if } |v|\leq R_2 \\ |v|  & \text{ if } |v|\geq 2R_2\end{cases}\\
r_3(x) &= \begin{cases} 1 & \text{ if } |u_i|\leq R_3 \\ |u|  & \text{ if } |u_i|\geq 2R_3\end{cases}
\end{split}
\end{align}
For $\delta \in \RR$, and weight function $\gamma > 0$, 
define the weighted H\"older space $C^{k,\alpha}_{\delta,\gamma}(M,T)$
of sections of any vector bundle $T$ over $M$ as the closure of
the space of $C^{\infty}$-sections in the norm 
\begin{align}
\begin{split}
&\Vert \sigma \Vert_{C^{k,\alpha}_{\delta,\gamma}(M,T)}  = \sum_{|\mathcal{I}|\leq k} |\gamma^{-\delta+|\mathcal{I}|}\nabla^\mathcal{I}\sigma|\\
&+ \sum_{|\mathcal{I}|=k}\sup_{0 < d(x,y) < \rho_{inj}}\Big( \min\{\gamma(x),\gamma(y)\}^{-\delta+k+\alpha}\frac{|\nabla^\mathcal{I}\sigma(x)-\nabla^\mathcal{I}\sigma(y)|}{d(x,y)^\alpha} \Big).
\end{split}
\end{align}
\begin{lemma}
Let $X_0$ be the LeBrun orbifold with quotient singularities $x_1, x_2, x_3$. The elliptic operator
\begin{align}
\label{Pdef}
P: C^{k,\alpha}_{\delta,r_1r_2}(X_0,&\Lambda^{0,1}\otimes\Theta)\xrightarrow{(\bar\partial^*,\bar\partial)} 
C^{k-1,\alpha}_{\delta-1,r_1r_2}(X_0,\Theta)\oplus C^{k-1,\alpha}_{\delta-1,r_1r_2}(X_0,\Lambda^{0,2}\otimes\Theta)
\end{align}
is Fredholm and surjective, where $\delta\in (-2,-1), k\geq 3$. 
\end{lemma}
\begin{proof}
First, note that that
any element of the kernel and cokernel is $\square$-harmonic. 
By the standard theory of harmonic functions, any $\square$-harmonic element which is $O(r_1^\delta)$ as $r_1 \rightarrow 0$ has a removable singularity.
The remainder of the proof is almost the same as the proof of \cite[Lemma~4.2]{HVI}, and is omitted. 
\end{proof}
We will now consider the weight function $\gamma : X \rightarrow \RR_+$ by the following: 
\begin{align}
\gamma &= \begin{cases} a^{-1}b^{-1}r_1r_2 & \text{ if } |z_i|\geq 2b \\ r_3 & \text{ if } |u_i|\leq 2a^{-1}\end{cases}.
\end{align}
Define the elliptic operator $P$ as $(\bar\partial^*,\bar\partial)$ with respect to the glued metric $\omega_{a,b}$, on the weighted space $C^{k,\alpha}_{\delta,\gamma}(X,\Lambda^{0,1}\otimes\Theta)$.
Because $\omega_{a,b}$ is a K\"ahler ALE metric, by \cite[Lemma~4.2]{HVI},
$P$ is Fredholm and surjective. 
Since each $P_{Y_i}$ has a bounded right inverse for $i = 1, 2, 3$, 
and $P_{X_0}$ has a bounded right inverse, a standard
argument (see for example \cite{RollinSinger}) shows that there is a uniformly bounded
right inverse of $P$, for $a, b$ sufficiently small. 

In \cite[Proposition 6.1]{LV14}, it was shown that $\dim(\ker(P_{X_0})) = b_\Gamma-1$, where $-b_\Gamma$ is the self-intersection number of the central divisor in $X$. For each Calderbank-Singer ALE space $Y_j$, by Lemma \ref{cyclic}, we have $\dim(\ker(P_{Y_j})) = \sum_{i=1}^{k_j}(e^j_{i}-1)$, where $-e^j_{i}$ is the self-intersection number of each irreducible exceptional divisor in $Y_j$. 

In the proof of \cite[Theorem~10.2]{HVI}, it was shown that the natural mapping 
\begin{align}
\label{natmap}
\H_{-3}(X,\Lambda^{0,1}\otimes\Theta) \rightarrow H^1(X,\Theta)
\end{align}
is surjective,
which implies that $\dim(H^1(X,\Theta))\leq \dim(\H_{-3}(X,\Lambda^{0,1} \otimes \Theta))$.
Also, from \eqref{es} above, we have $\dim(H^1(E,\oo_E(E)))\leq\dim(H^1(X,\Theta))$.
Combining these, we have that 
\begin{align}
\begin{split}
\dim(H^1(X,\Theta))&\leq \dim(\H_{-3}(X,\Lambda^{0,1}\otimes\Theta)) \\
&= \dim(\ker(P_{X})) =b_\Gamma-1+\sum_{j=1}^3\sum_{i=1}^{k_j}(e^j_{i}-1)\\
& = \dim(H^1(E,\oo_E(E)))\leq \dim(H^1(X,\Theta)).
\end{split}
\end{align}
This implies the isomorphism $H^1(X,\Theta)\simeq H^1(E,\oo_E(E))$. Then by the exact sequence \eqref{es}, $H^1(X,Der_E(X)) = 0$. 
\end{proof}

\section{Dimension of the moduli space}
\label{exsec}
We will next discuss the dimension of the generic orbit in the cases in Table \ref{dimtable}. 

\subsection{Discussion of Table \ref{dimtable}}
\label{distable}

Cases 1 and 2: $\Gamma = \frac{1}{p}(1,1)$. This case has been studied in \cite{HondaOn}. The group of biholomorphic automorphisms is ${\rm{GL}}(2,\CC)$, and the identity component of the holomorphic isometry group is ${\rm{U}}(2)$. When $p=3$, the action of ${\rm{U}}(2)$ coincides with the action of ${\rm{SU}}(2)$, and the dimension of generic orbits is $3$. Then $\dim(\mathfrak{M}) = 5-3 = 2$. When $p>3$, the dimension of each orbit is $4$, and $\dim(\mathfrak{M}) = 2p-5$.

\vspace{2mm}
\noindent
Case 3: $\Gamma = \frac{1}{p}(1,q)$, where $q\neq 1,p-1$. In this case, by direct calculation, the subgroup in ${\rm{U}}(2)$ that commutes with $\Gamma$ is isomorphic to $S^1\times S^1$. By \cite[Proposition~3.3]{ALM_resolution} the identity component of the 
holomorphic isometry group 
must be $S^1 \times S^1$. Using the fact that the cyclic quotient singularity is characterized by the invariant polynomials in \eqref{invpoly} above, it is easy to show that the dimension of the generic orbit of $\mathfrak{G}$ is $2$, and therefore $\dim(\mathfrak{M}) = j_\Gamma+k_\Gamma-2$.

\vspace{2mm}
\noindent
Case 4: $\Gamma$ is non-cyclic and not in ${\rm{SU}}(2)$. In this case, the subgroup in ${\rm{U}}(2)$ that commutes with $\Gamma$ is isomorphic to $S^1$, 
so by \cite[Proposition 3.3]{ALM_resolution} the identity component of the holomorphic isometry group must be $S^1$. Since the Hopf action is always nontrival on the normal bundle of the central divisor $E$, the dimension of the generic orbit of the Hopf action on 
$H^1(E, \mathcal{O}_E(E)) \cong H^0 (E, \mathcal{O}(b_{\Gamma} -2))$ 
is $1$, and therefore $\dim(\mathfrak{M}) = j_\Gamma+k_\Gamma-1$.

\subsection{Cyclic case}

Any cyclic action without complex reflections 
is conjugate to the action generated by 
\begin{align}
(z_1, z_2) \mapsto ( \xi_p z_1, \xi_p^q z_2),
\end{align}
where $\xi_p$ is a $p$th root of unity, and $q$ is relatively prime to $p$, which we will call a $\frac{1}{p} (1, q)$ action.
Define the integers $e_i \geq 2$, and $k$ by the continued fraction 
expansion 
\begin{align}
\label{fraction}
\frac{p}{q} = e_1 - \cfrac{1}{e_2 - \cdots \cfrac{1}{e_k}} \equiv [e_1, \dots, e_k].
\end{align}
The singularity of $\CC^2/\Gamma$ is known as a Hirzebruch-Jung singularity, and 
the exceptional divisor is a string 
of rational curves with normal crossing singularities. 

If $1 \leq q < p$, then let $q' = p -q$. Let 
$e_i' \geq 2$, and $k'$ denote integers arising in the 
the Hirzebruch-Jung algorithm for the $\frac{1}{p}(1, q')$-action.
In \cite{R1974}, Riemenschneider proved the formulas 
\begin{align}
\sum_{i = 1}^{k} (e_i -1) &= \sum_{i=1}^{k'} (e_i'-1),\\
k' &= e - 2,\\
e &= 3 + \sum_{i=1}^k (e_i -2),
\end{align}
where $e$ is the embedding dimension. 
In particular, these formulas give that 
\begin{align}
\sum_{i = 1}^{k} (e_i -1) = e + k - 3. 
\end{align}
From Subsection \ref{distable} above, for $q \neq 1, p-1$, 
it follows that
\begin{align}
m_{\Gamma} = 2 e_{\Gamma} + 3k_{\Gamma} - 8. 
\end{align}

\subsection{Non-cyclic cases}
The non-cyclic finite subgroups of ${\rm{U}}(2)$ without complex reflections
are given in Table \ref{groups},
where the binary polyhedral groups (dihedral, tetrahedral, octahedral, icosahedral) are 
respectively denoted by $D^*_{4n}$, $T^*$, $O^*$, $I^*$,
and the map 
$\phi:{\rm{SU}}(2) \times  {\rm{SU}}(2) \rightarrow {\rm SO}(4)$ denotes the usual double cover, see \cite{Brieskorn, BKR, LV14} for more details.

\begin{table}[ht]
\caption{Non-cyclic finite subgroups of ${\rm{U}}(2)$ containing no complex reflections}
\label{groups}
\begin{tabular}{ll l}
\hline
$\Gamma\subset{\rm U}(2)$ & Conditions & Order\\
\hline\hline
$ \phi(L(1,2 l) \times D^*_{4n}) $ &  $(l,2n) = 1$ & $4 l n$\\
$ \phi(L(1,2l) \times T^*) $  & $(l,6) = 1$ & $24l$\\
$  \phi(L(1,2l) \times O^*) $ & $(l,6) = 1$ & $48l$\\
$  \phi(L(1,2l)\times I^*) $ & $(l,30) = 1$ & $120l$ \\
 Index--$2$ diagonal $\subset   \phi(L(1,4l)\times D^*_{4n})$&  $(l,2) = 2,(l,n)=1$& $4ln$\\
 Index--$3$ diagonal $\subset  \phi(L(1,6l) \times T^*)$  & $(l,6)=3$ & $24l$.\\
\hline
\end{tabular}
\end{table}

First, consider the case of $\phi( L(1,2l) \times D^*_{4n})$, where 
$(l,2n) = 1$. In \cite{BKDihedral}, it is shown that 
\begin{align}
e_{\Gamma} = \sum_{i = 1}^{k} (e_i -2) + 3,
\end{align}
which implies that 
\begin{align}
\label{dform}
2 e_{\Gamma} + 3 k_{\Gamma} - 7 = 2 \sum_{i=1}^{k} (e_i -1) + k -1, 
\end{align}
which, from above, is equal to $m_{\Gamma}$.
For the index $2$ subgroup which is contained in 
$\phi(L(1,4l)\times D^*_{4n})$, where  $(l,2) = 2,(l,n)=1$,
a similar computation shows that the same formula 
\eqref{dform} holds in this case as well. 

\begin{table}[hb]
\caption{Cases with $T^*, O^*, I^*$ for $l > 1$}
\label{toi}
\begin{tabular}{ll}
\hline
$\Gamma\subset{\rm U}(2)$ &  $m_{\Gamma}$ \\
\hline\hline
$ \phi(L(1,2l) \times T^*) $  &   \\
 \hspace{3mm} $  l \equiv 1 \mod 6$  &  $\frac{1}{3} (l-1) + 17$   \\
 \hspace{3mm} $  l \equiv 5 \mod 6$  &  $\frac{1}{3} (l-5) + 15 $ \\
Index--$3$ diagonal $\subset  \phi(L(1,6l) \times T^*)$  & \\
\hspace{3mm} $(l,6) = 3$ &  $\frac{1}{3} (l-3) + 16$\\ 
$  \phi(L(1,2l) \times O^*) $ & \\
\hspace{3mm} $  l \equiv 1 \mod 12$  &  $\frac{1}{6} (l-1) + 20$   \\
\hspace{3mm} $  l \equiv 5 \mod 12$  &  $\frac{1}{6} (l-5) + 19$   \\
\hspace{3mm} $  l \equiv 7 \mod 12$  &  $\frac{1}{6} (l-7) + 18$   \\
\hspace{3mm} $  l \equiv 11 \mod 12$  &  $\frac{1}{6} (l-11) + 17$   \\
$  \phi(L(1,2l)\times I^*) $ &  \\
\hspace{3mm} $  l \equiv 1 \mod 30$  &  $\frac{1}{15} (l-1) + 23$  \\
\hspace{3mm} $  l \equiv 7 \mod 30$  &  $\frac{1}{15} (l-7) + 19$  \\
\hspace{3mm} $  l \equiv 11 \mod 30$  &  $\frac{1}{15} (l-11) + 22$  \\
\hspace{3mm} $  l \equiv 13 \mod 30$  &  $\frac{1}{15} (l-13) + 19$  \\
\hspace{3mm} $  l \equiv 17 \mod 30$  &  $\frac{1}{15} (l-17) + 18$  \\
\hspace{3mm} $  l \equiv 19 \mod 30$  &  $\frac{1}{15} (l-19) + 20$  \\
\hspace{3mm} $  l \equiv 23 \mod 30$  &  $\frac{1}{15} (l-23) + 18$  \\
\hspace{3mm} $  l \equiv 29 \mod 30$  &  $\frac{1}{15} (l-29) + 19$  \\
\hline
\end{tabular}
\end{table}

 Table \ref{toi} lists the dimension of the moduli space for 
subgroups of ${\rm{U}}(2)$ for finite subgroups involving 
$T^*, O^*, I^*$. The papers of \cite{Brieskorn, BKR, LV14} give a complete
description of the exceptional divisors and self-intersection numbers, 
and it is a straightforward computation to obtain the right column in the table, which, from Subsection~\ref{distable} is equal $j_{\Gamma} + k_{\Gamma} - 1$. Furthermore, using the formulas for the 
embedding dimension given in \cite{BKR}, it can be easily checked that
\begin{align}
m_{\Gamma} = j_{\Gamma} + k_{\Gamma} -1
=2 e_{\Gamma} + 3k_{\Gamma} - 7
\end{align}
in all of these cases (the computations are omitted). We point out that there is a typo in case $I_7$ in \cite{BKR}: in our notation, this case should have $l = 30(b-2) +7$ and $e = b +2$.

\subsection{Hyperk\"ahler case} 
Recall that a hyperk\"ahler metric is K\"ahler with respect to a $2$-sphere
of complex structures $S^2 = \{aI+bJ+cK: a^2+b^2+c^2=1\}$. 

\vspace{3mm}
\noindent
$A_1$-type:  the Eguchi-Hanson metric is an ALE Ricci-flat K\"ahler metric on $X = T^*S^2$. Since $k=1$, $B^1\subset \RR^2$, $B^2\subset \RR$, $d_\Gamma = 2+1 = 3$. 
With respect to the complex structure $I$ (the complex structure arising as
the total space of a holomorphic line bundle), 
the biholomorphic isometry group is ${\rm{U}}(2)$. 
The quotient $\mathfrak{F} / \mathfrak{G}$ has two orbit-types. 
The orbit of $(0, \rho)$ is $1$-dimensional. The 
orbit of $(t,\rho)$ where $t$ is non-zero, is also $1$-dimensional. 
Consequently,  $\mathfrak{M}$ is isomorphic to 
the $2$-dimensional upper half space. The remaining parameter of 
complex structures just corresponds to a hyperk\"ahler rotation, 
so the metrics obtained are all just scalings of the Eguchi-Hanson metric. 

Instead, consider the complex structure $J$. 
The biholomorphic isometry group is ${\rm{SU}}(2)$. The subspace $\VV$ 
is now of dimension $1$, so our parameter space is now
$\RR \times \RR$. The group $\mathfrak{G}$ now acts trivially,
so our parameter space is a ball in $\RR^2$. The remaining parameter of 
complex structures again just corresponds to a hyperk\"ahler rotation, 
so the metrics obtained are all just scalings of the Eguchi-Hanson metric.

\vspace{3mm}
\noindent
ADE-type: For the general $A_k, D_k, E_k$ $(k\geq 2)$ type ALE minimal resolution, the dimension of local moduli space of Ricci-flat K\"ahler metrics is $3k-3$.

For $A_k$ $(k\geq 2)$, this is the case of Gibbons-Hawking ALE hyperk\"ahler surface. $Aut(X) = \CC^*\times S^1\subset \RR_+ \times {\rm{U}}(2)$. Recall that ${\rm{U}}(2) = {\rm{U}}(1)\times_{\ZZ_2} {\rm{SU}}(2)$, which acts on $\vec{v}\in \CC^2$ as $g_L\cdot\vec{v}\cdot g_R$, where $g_L\in {\rm{U}}(1)$ is the left action and $g_R\in {\rm{SU}}(2)$ is the right action.
The $\CC^*$-action is generated by $(v_1,v_2)\rightarrow (\lambda v_1, \lambda v_2)$ where $\lambda\in \CC^*$; the $S^1$-action is generated by $(v_1,v_2)\rightarrow (\lambda v_1, \lambda^{-1} v_2))$ where $\vert \lambda\vert =1$.  The $\CC^*$ action induces a $2$-dimensional action on the hyperk\"ahler sphere, while the $S^1$ action preserves the hyperk\"ahler structure. Then $m_\Gamma = d_\Gamma-3 = 3k-3$.

For the case of $D_k, E_k$, $Aut(X) = \CC^*$. The $\CC^*$ action can be interpreted as follows: let $\mathfrak{g}_{\CC^*}$ denote the set of real vector fields which correspond to the Lie algebra of $\CC^*$. For any $Y\in\mathfrak{g}_{\CC^*}$, $\Phi^*_Y$ acts on the complex structures which gives an action on $B^1$. Since $Y$ is a real vector field, $\Phi^*_Y$ is transverse to the action on the hyperk\"ahler sphere (it is not transverse only in the $A_k$ case). Then the dimension of the maximal orbit generated by the $\CC^*$ action and the action on hyperk\"ahler sphere is $3$, so $m_\Gamma = d_\Gamma-3$.

\bibliography{Han_Viaclovsky}

\end{document}